\documentclass[10pt,leqno]{amsart}
\usepackage{indentfirst, amssymb,amsmath,amsthm}
\newtheoremstyle{case}{}{}{}{}{}{:}{ }{}
\newtheoremstyle{subcase}{}{}{}{}{}{:}{ }{}

\newtheorem*{theoA}{Theorem A}
\newtheorem*{theoB}{Theorem B}
\newtheorem*{theoC}{Theorem C}
\newtheorem*{theoD}{Theorem D}
\newtheorem*{theoE}{Theorem E}
\newtheorem*{theoF}{Theorem F}

\newtheorem{theo}{Theorem}[section]
\newtheorem{lem}{Lemma}[section]
\newtheorem{cor}{Corollary}[section]
\newtheorem{note}{Note}[section]

\newtheorem{defi}{Definition}[section]
\newtheorem{rem}{Remark}[section]
\newtheorem{ques}{Question}[section]
\newcommand{\ol}{\overline}
\newcommand{\ul}{\underline}

\numberwithin{subcase}{case}
\newcommand{\be}{\begin{equation}}
\newcommand{\ee}{\end{equation}}
\newcommand{\beas}{\begin{eqnarray*}}
	\newcommand{\bea}{\begin{eqnarray}}
	\newcommand{\eea}{\end{eqnarray}}
	\newcommand{\eeas}{\end{eqnarray*}}

\newcommand{\bd}{\begin{doublespacing}}
	\newcommand{\ed}{\end{doublespacing}}
\usepackage{amssymb,setspace,amsbsy,indentfirst}
\numberwithin{equation}{section}
\renewcommand{\vline}{\mid}
\begin{document}
	\title[Uniqueness of certain differential polynomial...]{Uniqueness of certain differential polynomial of $L$-functions and meromorphic functions sharing a polynomial}
	\numberwithin {equation}{section}
	\date{}
	\author{Abhijit Banerjee\;\;\; And \;\;\;Saikat Bhattacharyya}
	\date{}
	\address{ Department of Mathematics, University of Kalyani, West Bengal 741235, India.}
	\email{abanerjee\_kal@yahoo.co.in, abanerjeekal@gmail.com}
	
	\address{Department of Science and Humanities, Jangipur Government Polytechnic, West Bengal 742225, India.}
	\email{saikat352@gmail.com, saikatbh89@yahoo.com}
	\
	\maketitle
	\let\thefootnote\relax
	\footnotetext{Type set by \AmS -\LaTeX}
	\begin{abstract}
The purpose of this paper is to obtain some sufficient conditions to determine the relation between a meromorphic function and an $L$-function when certain differential polynomial generated by them sharing a one degree polynomial. The main theorem of the paper extends and improves all the results in \cite{DDNS_HaoChe_18}, \cite{PJASA_LiuLiYi_17} and \cite{TMJ_SahHal_18}.    
	   \end{abstract}
	\section{Introduction Definitions and Results}
In this paper, we use the term ``$L$-function" to denote a Selberg class function that are Dirichlet series with the Riemann zeta function $\zeta(s)= \sum\limits_{n=1}^{\infty}n^{-s}$ as the prototype. In the beginning of the nineteenth century R. Nevanlinna inaugurated the value distribution theory with his famous Five Value and Four Value theorems which were the bases of uniqueness theory. Value distribution of $L$-functions concerns distribution of zeros of $L$-functions and more generally, the $c$-points of $\mathcal{L}$, that is, the zeros of the function $\mathcal{L}(s)-c$, or the values in the set of pre-images $L^{-1}= \{s \in \mathbb{C}: \mathcal{L}(s)= c\}$, where $s$ denotes complex variables and $c\in \mathbb{C}\cup \{\infty\} $. Selberg class functions are important objects in number theory. The Selberg class $S$ of $L$-functions is the set of all Dirichlet series $\mathcal{L}(s)= \sum\limits_{n=1}^{\infty}a(n)n^{-s}$ of a complex variable $s=\sigma+ it$ with $a(1)=1$, satisfying the following axioms (cf. \cite{PACANT_Sel_92}, \cite{LNM_Ste_07}):\par 
(i) Ramanujan hypothesis: $a(n)\ll n^{\epsilon}$ for every $\epsilon > 0$.\par 
(ii) Analytic continuation: There is a nonnegative integer $k$ such that $(s-1)^k \mathcal{L}(s)$ is an entire function of finite order.\par
(iii) Functional equation: $\mathcal{L}$ satisfies a functional equation of type $\Lambda_\mathcal{L}(s)= \omega \Lambda_\mathcal{L}(1-\ol s)$, where $\Lambda_\mathcal{L}(s)= \mathcal{L}(s)Q^s\prod\limits_{j=1}^{K}\Gamma(\lambda_js+v_j)$ with positive real numbers $Q$, $\lambda_j$ and complex numbers $v_j$, $\omega$ with $Rev_j\geq 0$ and $|\omega|=1$. \par 
(iv)Euler product hypothesis: $\mathcal{L}(s)= \prod_p\exp\bigg(\sum\limits_{k=1}^{\infty}\frac{b(p^k)}{p^{ks}}\bigg)$ with suitable coefficients $b(p^k)$ satisfying $b(p^k)\ll p^{k\theta}$ for some $\theta < \frac{1}{2}$, where the product is taken over all prime numbers $p$.\par 
The degree $d$ of an $L$-function $\mathcal{L}$ is defined to be 
\beas d= 2\sum\limits_{j=1}^{K}\lambda_j,\eeas
where $K$ and $\lambda_j$ are respectively the positive integer and the positive real number defined in axiom (iii) of the definition of $L$-function.\par
It is to be noted that an $L$-function can be analytically continued as meromorphic function in $\mathbb{C}$.\par
	Throughout the paper, the term ``meromorphic" will be used to mean meromorphic in the whole complex plane. For such two meromorphic functions $f$, $g$ and for some $a\in\mathbb{C}$, we denote by $E(a;f)$, the collection of the zeros of $f-a$, where a zero is counted according to its multiplicity. In addition to this, when $a=\infty$, the above definition implies that we are considering the poles. In the same manner, by $\ol E(a;f)$, we denote the collection of the distinct zeros or poles of $f-a$ according as $a\in\mathbb{C}$ or $a=\infty$ respectively. If $E(a;f)=E(a;g)$ we say that $f$ and $g$ share the value $a$ CM (counting multiplicities) and if  $\ol E(a;f)=\ol E(a;g)$, then we say that $f$ and $g$ share the value $a$ IM (ignoring multiplicities). 
Usually, $S(r,f)$ denotes any quantity satisfying $S(r,f)= o(T(r,f))$ for all $r$ outside of a possible exceptional set of finite linear measure.\par 
For a meromorphic function $f$, we define the order $\rho(f)$ as:
\beas \rho(f)= \limsup\limits_{r \to \infty}\frac{\log T(r,f)}{\log r}.\eeas

In 1997, Lahiri \cite{YMJ_Lah_97} asked the following question:
\begin{ques}\cite{YMJ_Lah_97} What can be said about the relationship between two meromorphic functions $f$ and $g$ when two differential polynomials generated by them share some non-zero complex values? \end{ques}
In response to the above question plethora of investigations have been carried out on differential polynomials sharing non-zero complex values and even sets.\par 
Recently Liu-Li-Yi \cite{PJASA_LiuLiYi_17} carry forwarded the above investigations and explored over the uniqueness property of $L$-function and any meromorphic function when two differential polynomials generated by both of them share any finite complex value. Liu-Li-Yi \cite{PJASA_LiuLiYi_17} obtained the following result.
\begin{theoA} \cite{PJASA_LiuLiYi_17}
Let $f$ be a non-constant meromorphic function, $\mathcal{L}$ be an $L$-function and $n$, $k$ be two positive integers such that $n > 3k+6$. If $(f^n)^{(k)}-\alpha(z)$ and $(\mathcal{L}^n)^{(k)}-\alpha(z)$ share $(0, \infty)$, then $f=t\mathcal{L}$ for a constant $t$ satisfying $t^n=1$, where $\alpha(z)$ is either $1$ or $z$.
\end{theoA}
In 2001, the introduction of the notion of weighted sharing \cite{NMJ_Lah_01, CVTA_Lah_01}, of values and sets, which is actually a scaling between CM and IM sharing, further add essence to the uniqueness literature. Below we invoke the definition.
\begin{defi} \cite{NMJ_Lah_01, CVTA_Lah_01} Let $k$ be a non-negative integer or infinity. For $a\in\mathbb{C}\cup\{\infty\}$ we denote by $E_{k}(a;f)$ the set of all $a$-points of $f$, where an $a$-point of multiplicity $m$ is counted $m$ times if $m\leq k$ and $k+1$ times if $m>k$. If $E_{k}(a;f)=E_{k}(a;g)$, we say that $f, \; g$ share the value $a$ with weight $k$ and denote it by $(a,k)$. The IM and CM sharing corresponds to $(a,0)$ and $(a,\infty)$ respectively.\end{defi}
We also say that $f(z)$ and $g(z)$ share a polynomial $p(z)$ with weight $l$ if $f(z)-p(z)$ and $g(z)-p(z)$ share $(0,l)$. \par
In 2018, Sahoo-Halder \cite{TMJ_SahHal_18} employed the notion of weighted sharing of values to relax the nature of sharing of value in the above theorem as follows:
\begin{theoB} \cite{TMJ_SahHal_18}
Under the same situation as in {\it Theorem A}, if $(f^n)^{(k)}-\alpha(z)$ and $(\mathcal{L}^n)^{(k)}-\alpha(z)$ share $(0, l)$ and one of the following conditions is satisfied: (i) $l \geq 2$ and $n>3k+6$, (ii) $l=1$ and $n> \frac{7k}{2}+ \frac{13}{2}$, (iii) $l=0$ and $n> 7k+11$, then $f=t\mathcal{L}$ for some constant $t$ satisfying $t^n=1$, where $\alpha(z)$ is either $1$ or $z$.
\end{theoB}
In the same year, Hao-Chen \cite{DDNS_HaoChe_18} generalized the differential polynomials generated by meromorphic function $f$ and $L$-function $\mathcal{L}$ to obtain a series of following four theorems.
\begin{theoC} \cite{DDNS_HaoChe_18}
Let $f$ be a non-constant meromorphic function, $\mathcal{L}$ be an $L$-function, $n$, $m$, $k$ be  three positive integers and $\alpha$, $\beta$ be two constants satisfying $|\alpha|+ |\beta|\not=0$. Suppose that $[f^n(\alpha f^m+ \beta)]^{(k)}$ and $[\mathcal{L}^n(\alpha \mathcal{L}^m+ \beta)]^{(k)}$ share $(1, \infty)$. If $n> 3k+\tilde{m}+6$, then $f= t \mathcal{L}$, where \par 
(i) when $\alpha \beta =0$, $t$ is a constant such that $t^{n+\tilde{m}}=1$,\par 
(ii) when $\alpha \beta \not =0$, $k \geq 2$, $t$ is a constant such that $t^d=1$.\par
Here $d= GCD(n,m)$ and $\tilde{m}:= \tilde{m}(\alpha)$, where 
\beas \tilde{m}(\alpha)= \begin{cases}
	0, \;\;\;\; \alpha=0,\\ m, \;\;\;\; \alpha\not=0.
\end{cases} \eeas 
\end{theoC}
\begin{theoD} \cite{DDNS_HaoChe_18}
Let $f$ be a non-constant meromorphic function, $\mathcal{L}$ be an $L$-function and $n$, $m$, $k$ be three positive integers. Suppose $[f^n(f-1)^m]^{(k)}$ and $[\mathcal{L}^n(\mathcal{L}-1)^m]^{(k)}$ share $(1, \infty)$. If $n> 3k+m+6$ and $k \geq 2$, then $f\equiv \mathcal{L}$ or $f^n (f-1)^m \equiv \mathcal{L}^n(\mathcal{L}-1)^m$. 
\end{theoD}
\begin{theoE} \cite{DDNS_HaoChe_18} 
Under the same situation as in {\it Theorem C}, $[f^n(\alpha f^m+ \beta)]^{(k)}$ and $[\mathcal{L}^n(\alpha \mathcal{L}^m+ \beta)]^{(k)}$ share $(1, 0)$ and $n> 7k+4\tilde{m}+11$, then $f= t \mathcal{L}$, where \par 
(i) when $\alpha \beta =0$, $t$ is a constant such that $t^{n+\tilde{m}}=1$,\par 
(ii) when $\alpha \beta \not =0$, $k \geq 2$, $t$ is a constant such that $t^d=1$.\par 
\end{theoE}
\begin{theoF} \cite{DDNS_HaoChe_18}
Under the same situation as in {\it Theorem D}, if $[f^n(f-1)^m]^{(k)}$ and $[\mathcal{L}^n(\mathcal{L}-1)^m]^{(k)}$ share $(1, 0)$ and $n> 7k+4m+11$, $k \geq 2$, then $f\equiv \mathcal{L}$ or $f^n (f-1)^m \equiv \mathcal{L}^n(\mathcal{L}-1)^m$.
\end{theoF}
\begin{note}
The differential polynomial in {\it Theorems D and F} becomes identical with that in {\it Theorems C and E} for $m=1$, so the condition $m \geq 2$ is required in {\it Theorems D and F}.
\end{note}
\begin{rem}
Here we would like to mention that very recently  Li-Liu-Yi \cite{Filomat_LiLiuYi_19} obtained {\it Theorem D} and {\it Theorem F} for $m=1$ when $n>3k+9$ and $n>7k+17$ respectively. Hence the results are insignificant in context to the lower-bound of $n$ in {\it Theorem D} and {\it Theorem E}. 
\end{rem} 
The purpose of the paper is to bring all the above results under a single umbrella. To this end, we consider a more generalized differential polynomial generated by a meromorphic function and an $L$-function and significantly improve all the above results.\par 
Throughout the paper let us denote by $P(z)$ the following $n$ degree polynomial:
\bea \label{e1.1} P(z)= \sum\limits_{j=1}^{n} a_j z^j = a_n \prod\limits_{j=1}^{s}(z-d_{l_j})^{l_j},\eea
where $a_1, \ldots,a_n(\not=0) \in \mathbb{C}$ and $d_{l_j}$ (j=1, 2, \ldots, s) are distinct and $l_1, l_2, \ldots, l_s$, $s,\; n \in \mathbb{N}$ such that $\sum\limits_{j=l}^{s}l_j=n$.  Clearly, $P(0)=0$.\par
We denote by $n_1$ and $n_2$ respectively be the number of simple and multiple zeros of $P(z)$. \par 
The main result of the paper is given below. We shall show that the corollaries deduced from the main result will improve {\it Theorems B-F} by reducing the lower bound of $n$.\par  
Throughout the paper we will use $\eta(z)=az+b$, where $\mid a \mid + \mid b \mid \not = 0$. 
\begin{theo} \label{t1}
	Let $f$ be a non-constant meromorphic function, $\mathcal{L}$ be an $L$-function, $s$ be a non-negative integer, $n$, $m$, $k$ be three positive integers and $\alpha$, $\beta$ be two constants with $\mid \alpha \mid+ \mid \beta \mid \not=0$. Suppose that $[P(f)(\alpha f^m+ \beta)^s]^{(k)}-\eta(z)$ and $[P(\mathcal{L})(\alpha \mathcal{L}^m+ \beta)^s]^{(k)}-\eta(z)$ share $(0,l)$. If \par
	$l\geq 2$ and  
	\bea \label{e1.2} n>\frac{k}{2}+2+2n_2(k+2) +2n_1+ms,\eea
	or, $l=1$ and 
	\bea \label{e1.3} n>\frac{3k}{4}+\frac{9}{4}+\bigg(\frac{5k}{2}+\frac{9}{2}\bigg)n_2 +\frac{5n_1}{2}+\frac{3ms}{2},\eea 
	or, $l=0$ and  
	\bea \label{e1.4} n>2k+\frac{7}{2}+(5k+7)n_2 +5n_1+4ms,\eea then one of the following two cases holds.\par
	 (i) $[P(f)(\alpha f^m+ \beta)^s]^{(k)}[P(\mathcal{L})(\alpha \mathcal{L}^m+ \beta)^s]^{(k)}=\eta^2(z)$;\par
	 (ii) $P(f)(\alpha f^m+\beta)^s=P(\mathcal{L})(\alpha \mathcal{L}^m+\beta)^s$ or $f = t\mathcal{L}$, for a constant $t$ satisfying $t^{\chi_n}=1$, where
	 \beas \chi_n =\begin{cases}
	 	1, & \displaystyle\sum_{j=1}^{n-1}\mid a_{n-j}\mid\not=0; \\ d_1, & a_j=0, \forall j=1,2,\ldots, n-1,	\end{cases}\eeas
	 $d_1=\gcd(ms+n,\ldots, m(s-i)+n, \ldots n)$, $i=0,1,\ldots,s$.
\end{theo}
Putting $s=0$ and $P(z)=z^n$ in {\it Theorem \ref{t1}}, we obtain the following corollary which improves {\it Theorem B} by reducing the lower bound $n$.
\begin{cor} \label{c1} 
Let $f$ be a non-constant meromorphic function, $\mathcal{L}$ be an $L$-function and $n$, $k$ be two positive integers. Suppose that $(f^n)^{(k)}-\eta(z)$ and $(\mathcal{L}^n)^{(k)}-\eta(z)$ share $(0,l)$. If \par $l\geq 2$ and
\bea \label{e1.5} n>\frac{5k}{2}+6, \eea
or, $l=1$ and 
\bea \label{e1.6} n>\frac{13k}{4}+\frac{27}{4},\eea
  or, $l=0$ and 
  \bea \label{e1.7} n > 7k+ \frac{21}{2},\eea 
   then $f = t\mathcal{L}$ for a constant $t$ satisfying $t^n=1$. 
\end{cor}
Putting $s=1$ and $P(z)=z^n$ in {\it Theorem \ref{t1}} we obtain the following corollary which improves {\it Theorem C and E} by reducing the lower bound $n$. 
\begin{cor} \label{c2}
Let $f$ be a non-constant meromorphic function, $\mathcal{L}$ be an $L$-function, $n$, $m$, $k$ be three positive integers and $\alpha$, $\beta$ be two constants such that $\mid \alpha \mid+ \mid \beta \mid \not=0$. Suppose that $[f^n(\alpha f^m+ \beta)]^{(k)}-\eta(z)$ and $[\mathcal{L}^n(\alpha \mathcal{L}^m+ \beta)]^{(k)}-\eta(z)$ share $(0,l)$. If \par 
$l\geq 2$ and
\bea \label{e1.8} n>\frac{5k}{2}+m+6,\eea
or, $l=1$ and 
\bea \label{e1.9} n>\frac{13k}{4}+\frac{3m}{2}+\frac{27}{4}, \eea
or, $l=0$ and 
\bea \label{e1.10} n > 7k+4m+ \frac{21}{2}, \eea
 then one of the following two cases holds. \par  
(i) when $\alpha \beta =0$, then $f = t \mathcal{L}$, for a constant $t$ satisfying $t^{n+\tilde{m}}=1$;\par 
(ii)  when $\alpha \beta \not=0$ and $k \geq 2$, then $f \equiv t \mathcal{L}$, $t$ is a constant satisfying $t^d=1$.
\end{cor}
 
Putting $m=1$, $\alpha=1$, $\beta=-1$ and $P(z)=z^n$ in {\it Theorem \ref{t1}}, we obtain the following corollary which again improves {\it Theorem D and F} by reducing the lower bound of $n$.
\begin{cor} \label{c3}
	Let $f$ be a non-constant meromorphic function, $\mathcal{L}$ be an $L$-function, $s$ be a non-negative integer and $n$, $k(\geq 2)$ be two positive integers. Suppose that $[f^n(f-1)^s]^{(k)}-\eta(z)$ and $[\mathcal{L}^n(\mathcal{L}-1)^s]^{(k)}-\eta(z)$ share $(0,l)$. If \par 
	$l\geq 2$ and 
	\bea \label{e1.11} n>\frac{5k}{2}+s+6,\eea
	or, $l=1$ and 
	\bea \label{e1.12} n>\frac{13k}{4}+\frac{3s}{2}+\frac{27}{4}, \eea
	or, $l=0$ and 
	\bea \label{e1.13} n>7k+4s+\frac{21}{2},\eea
	 then either $f \equiv \mathcal{L}$ or $f^n (f-1)^s \equiv \mathcal{L}^n(\mathcal{L}-1)^s$.
\end{cor}
For the standard definitions and notations of the value distribution theory we refer to \cite{Clarendon_Hay_64}. But in the paper we have used some more notations and definitions which are explained below.
\begin{defi} \cite{KMJ_Yi_99} Let $f$ and $g$ be two non-constant meromorphic functions such that $f$ and $g$ share $(a, 0)$. Let $z_0$ be an $a$-point of $f$ with multiplicity $p$, an $a$-point of $g$ with multiplicity $q$. We denote by $\ol N_L(r,a;f)$ the reduced counting of those $a$-points of $f$ and $g$ where $p>q$, by $N_E^{1)}(r,a;f)$ the counting function of those $a$-points of $f$ and $g$ where $p=q=1$, by $\ol N_E^{(2}(r,a;f)$ the reduced counting function of those $a$-points of $f$ and $g$ where $p=q \geq 2$. In the same way we can define $\ol N_L(r,a;g)$, $N_E^{1)}(r,a;g)$, $\ol N_E^{(2}(r,a;g)$. In a similar manner we can define  $\ol N_L(r,a;f)$ and $\ol N_L(r,a;g)$ for $a\in\mathbb{C}\cup\{\infty\}$.  \end{defi}
When $f$ and $g$ share $(a,m)$, $m \geq 1$, then $N_E^{1)}(r,a;f)= N(r,a;f \mid =1)$.
\begin{defi} \cite{NMJ_Lah_01, CVTA_Lah_01} 
	Let $f$, $g$ share a value $(a,0)$. We denote by $\ol N_*(r,a;f,g)$ the reduced counting function of those $a$-points of $f$ whose multiplicities differ from the multiplicities of the corresponding $a$-points of $g$. \par 
	Clearly, $\ol N_*(r,a;f,g)= \ol N_*(r,a;g,f)= \ol N_L(r,a;f)+ \ol N_L(r,a;g)$. 
\end{defi}
\begin{defi} \cite{IJMS_Lah_01} For $a\in\mathbb{C}\cup\{\infty\}$ and for a positive integer $p$ we denote by $N(r,a;f\vline\leq p) (N(r,a;f\vline\geq p))$ the counting function of those $a$-points of $f$ whose multiplicities are not greater(less) than $p$ where each $a$-point is counted according to its multiplicity.
	
	$\ol N(r,a;f\vline\leq p) (\ol N(r,a;f\vline\geq p))$ are defined similarly, where in counting the $a$-points of $f$ we ignore the multiplicities.
	
	Also $N(r,a;f\vline <p)$, $N(r,a;f\vline >p)$, $\ol N(r,a;f\vline <p)$ and $\ol N(r,a;f\vline >p)$ are defined analogously.  \end{defi}
\begin{defi} \cite{CVTA_Lah_01}
Let $p$ be positive integer or infinity. We denote by $N_p(r,a;f)$ the counting function of $a$-points of $f$, where an $a$-point of multiplicity $m$ is counted $m$ times if $m\leq p$ and $p$ times if $m>p$. Then 
\beas N_p(r,a;f)= \ol N(r,a;f)+\ol N(r,a;f\vline \geq 2)+ \ldots + \ol N(r,a;f\vline \geq p).\eeas
Clearly, $N_1(r,a;f)= \ol N(r,a;f).$
\end{defi}
\begin{defi} \label{d7}
Let $a$ be any value in the extended complex plane and let $k$ be an arbitrary non-negative integer. We define 
\beas \Theta(a,f)=1-\limsup\limits_{r \to \infty}\frac{\ol N(r,a;f)}{T(r,f)} \eeas	
and 
\beas \delta_k(a,f) = 1= \limsup\limits_{r \to \infty}\frac{N_k(r,a;f)}{T(r,f)}. \eeas
\end{defi}
	\begin{rem} \label{r1}
	From the definitions  of $\Theta(a,f)$ and $\delta(a,f)$ we clearly see that \beas 0 \leq \delta_{k} (a,f) \leq \delta_{k-1}(a,f)\leq \delta_1(a,f) \leq \Theta(a,f)\leq 1.\eeas
\end{rem}

\section{Lemmas}
 Let for two non-constant meromorphic functions $F$ and $G$ we denote by $H$ the following function: 
\bea \label{e2.1} H&=&\left(\frac{\;\;F^{''}}{F^{'}}-\frac{2F^{'}}{F-1}\right)-\left(\frac{\;\;G^{''}}{G^{'}}-\frac{2G^{'}}{G-1}\right).\eea
\begin{lem} \cite{MathZ_Yang_72} \label{l1} 
Suppose that $f$ is a non-constant meromorphic function and let $a_0, a_1, \ldots, a_n$ be finite complex numbers such that $a_n\not=0$. Then \beas T(r, a_n f^n+ a_{n-1}f^{n-1}+ \ldots+ a_1 f+ a_0)= nT(r,f)+ S(r,f). \eeas
\end{lem}
\begin{lem} \cite{KMJ_Yi_99} \label{l2} Let $F$, $G$ be two non-constant meromorphic functions such that they share $(1,0)$ and $H \not \equiv 0$ then 	\beas N_E^{1)}(r,1;F) \leq N(r,\infty; H)+ S(r,F)+ S(r,G).\eeas \end{lem}

\begin{lem} \cite{TJM_Ban_10} \label{l3} Let $F$, $G$ be two non-constant meromorphic functions sharing $(1,l)$, where $ 0\leq l < \infty $. Then \beas \ol N(r,1;F)+ \ol N(r,1;G)- N_E^{1)}(r,1;F)+\bigg(l-\frac{1}{2}\bigg)\ol N_*(r,1;F,G) \leq \frac{1}{2}[N(r,1;F)+ N(r,1;G)].\eeas \end{lem}

\begin{lem}\cite{JIPAM_ZhaYan_07} \label{l4} 
Let $f$ be a non-constant meromorphic function and $k$, $p$ are positive integers. Then 
\beas N_p(r, 0; f^{(k)}) \leq T(r, f^{(k)})- T(r,f)+ N_{p+k}(r,0; f)+ S(r,f),\eeas
\beas N_p(r, 0; f^{(k)}) \leq k \ol N(r, \infty;f)+ N_{p+k}(r,0; f)+ S(r,f).\eeas   
\end{lem}
\begin{lem} \label{l4a}
	Let $F$, $G$ be two non-constant meromorphic functions such that they share $(1,l)$. Then 
	\beas \ol N_*(r,1;F,G) \leq \frac{1}{l+1}\{\ol N(r,0;F)+ \ol N(r,0;G)+ \ol N(r, \infty; F) + \ol N(r,\infty; G)\}+ S(r,F)+ S(r,G).\eeas 
\end{lem}
\begin{proof}
	The proof can be carried out in the line of the proof of {\it Lemma 2.6} in \cite{RCMP_BanBha_16}.
\end{proof}
\begin{lem} \cite{Clarendon_Hay_64} \label{l5}
Let $f$ be a non-constant meromorphic function, $k$ be a positive integer and let $c$ be a non-zero finite complex number. Then 
\beas T(r,f) &\leq& \ol N(r,\infty; f) + N(r,0;f)+ N(r,c; f^{(k)})- N(r, 0; f^{(k+1)})+ S(r,f)\\ &\leq & \ol N(r,\infty; f)+ N_{k+1}(r, 0; f)+ \ol N(r,c;f^{(k)})- N_0(r,0;f^{(k+1)})+ S(r,f),\eeas
where $N_0(r,0;f^{(k+1)})$ is the counting function of those zeros of $f^{(k+1)}$ in $|z| < r$ which are not zeros of $f(f^{(k)}-c)$ in $|z|<r$.
\end{lem}

\begin{lem} \cite{SSSA_Yan_86} \label{l7} Let $f$ be a non-constant meromorphic function, $\alpha(\not\equiv 0, \infty)$ be a small function of $f$. Then \beas T(r,f) \leq \ol N(r,\infty; f)+ N(r,0;f)+ N(r, 0; f^{(k)}-\alpha)- N\bigg(r, 0; \bigg(\frac{f^{(k)}}{\alpha}\bigg)'\bigg)+ S(r,f).\eeas \end{lem}
\begin{lem} \cite{BLMS_Lan_03} \label{l8}
Suppose that $f$ is meromorphic of finite order in the complex plane and that $f^{(k)}$ has finitely many zeros, for some $k \geq 2$. Then $f$ has finitely many poles in the complex plane. 
\end{lem}
\begin{lem} \cite{KMJ_Yi_99} \label{l9a}
If $H \equiv 0$, then $F$, $G$ share $(1, \infty).$ If further $F$, $G$ share $(\infty,0)$ then $F$, $G$ share $(\infty, \infty)$.
\end{lem}
\begin{lem} \label{l10}
Let $f$ and $g$ be two transcendental meromorphic functions and $H \not\equiv 0$. Let for  two integers $k(\geq 1)$ and $l(\geq 0)$, $f^{(k)}-Q$, $g^{(k)}-Q$ share $(0,l)$, where $Q \not\equiv 0$ is a polynomial. Then
\beas \frac{1}{2}[T(r, f)+ T(r,g)] &\leq& \bigg(\frac{k}{2}+2\bigg)\bigg[\ol N(r, \infty;f)+\ol N(r, \infty;g)\bigg]+ N_{k+2}(r, 0; f)+N_{k+2}(r, 0; g)\\&&-\bigg(l-\frac{3}{2}\bigg) \ol N_*(r, 1; F, G)+ S(r,f)+ S(r,g),  \eeas
where $F=\frac{f^{(k)}}{Q}$ and $G=\frac{g^{(k)}}{Q}$.
\end{lem}
\begin{proof}
	Since $f$ and $g$ are two transcendental meromorphic functions, $F$ and $G$ are also two transcendental meromorphic functions.
	Let $z_0$ is a common simple zero of $f^{(k)}-Q$ and $g^{(k)}-Q$. Then $z_0$ is a common simple zero of $F-1$ and $G-1$. We can easily verify that possible pole of $H$ occur at (i)  multiple zeros of $F$ and $G$, (ii) poles of $f$ and $g$, (iii) $1$-points of $F$ and $G$ of different multiplicities, (iv) zeros of $F'$ which are not the zeros of $F(F-1)$, (v) zeros of $G'$ which are not the zeros of $G(G-1)$. Since $H$ has only simple poles, clearly  we have 
	\bea \label{e2.2} && N(r, \infty; H)\\ &\leq& \ol N(r, \infty; f) + \ol N(r, \infty;g)+ \ol N(r, 0;F \mid \geq 2)+\ol N(r, 0 ; G \mid \geq 2)+ \ol N_*(r, 1; F,G)\nonumber\\&&+ \ol N_{\otimes}(r, 0; F')+ \ol N_{\otimes}(r, 0; G')+ O(\log r)\nonumber,\eea
	where $\ol N_{\otimes}(r, 0; F')$ denotes the reduced counting function of those zeros of $F'$ which are not the zeros of $F(F-1)$ and $\ol N_{\otimes}(r, 0; G')$ is similarly defined.\par Since $f$ is a transcendental meromorphic function, we have 
	\beas T(r, Q)= o\{T(r,f)\}.\eeas
	
	By using {\it Lemma \ref{l7}}, we get 
	\bea \label{e2.3}&& T(r,f)+ T(r,g)\\ &\leq& \ol N(r,\infty; f)+ N(r, 0; f)+ N(r,1;F)+ \ol N(r,\infty; g)+ N(r, 0; g)\nonumber\\&&+ \ol N(r,1;G)- N(r,0;F')- N(r,0;G')+ S(r,f)+ S(r,g)\nonumber \\  &\leq& \ol N(r,\infty; f)+ N_{k+1}(r, 0; f)+ \ol N(r,1;F)+ \ol N(r,\infty; g)+ N_{k+1}(r, 0; g)\nonumber\\&&+ \ol N(r,1;G)- N_0(r,0;F')- N_0(r,0;G')+ S(r,f)+ S(r,g)\nonumber,\eea
	where $N_0(r,0;F^{'})$ is the counting function of those zeros of $F^{'}$ in $|z| < r$ which are not the zeros of $f(F-1)$ in $|z|<r$.\par 
	Now using {\it Lemma \ref{l1}, \ref{l2}, \ref{l3}} and (\ref{e2.2}), we get
	\bea \label{e2.4} &&\ol N(r, 1; F)+ \ol N(r, 1; G)\\ &\leq& \frac{1}{2}[N(r, 1; F)+ N(r, 1; G)]+ N_E^{1)}(r, 1;
	F)- \bigg(l-\frac{1}{2}\bigg) \ol N_*(r, 1; F,G)\nonumber\\ &\leq& \frac{1}{2} \{T(r,f)+ T(r,g)\}+ \bigg(\frac{k}{2}+1\bigg)\ol N(r, \infty;f)+\bigg(\frac{k}{2}+1\bigg)\ol N(r, \infty;g)+  \ol N(r,0; F\mid \geq 2)\nonumber\\&&+ \ol N(r, 0; G \mid \geq 2)-\bigg(l-\frac{3}{2}\bigg) \ol N_*(r, 1; F, G)+ \ol N_{\otimes}(r, 0; F')+ \ol N_{\otimes}(r, 0; G')+ O(\log r)\nonumber. \eea
	So from (\ref{e2.3}) and (\ref{e2.4}), we obtain
	\beas && \frac{1}{2}[T(r, f)+ T(r,g)]\\ &\leq&  \bigg(\frac{k}{2}+2\bigg)\bigg[\ol N(r, \infty;f)+\ol N(r, \infty;g)\bigg]+ N_{k+1}(r, 0; f)+ \ol N(r,0; F\mid \geq 2)\nonumber\\&& +N_{k+1}(r, 0; g)+ \ol N(r, 0; G \mid \geq 2)-\bigg(l-\frac{3}{2}\bigg) \ol N_*(r, 1; F, G)+ \ol N_{\otimes}(r, 0; F')\\&&+ \ol N_{\otimes}(r, 0; G')- N_0(r,0;F')- N_0(r,0;G')+ S(r,f)+ S(r,g)\\ &\leq& \bigg(\frac{k}{2}+2\bigg)\bigg[\ol N(r, \infty;f)+\ol N(r, \infty;g)\bigg]+ N_{k+1}(r, 0; f)+ \ol N(r,0; f\mid \geq k+2)\\&&+ \ol N(r,0; F\mid \geq 2\mid f\not=0)+N_{k+1}(r, 0; g)+ \ol N(r,0; g\mid \geq k+2)+\ol N(r,0; G\mid \geq 2\mid g\not=0)\\&&-\bigg(l-\frac{3}{2}\bigg) \ol N_*(r, 1; F, G)+ \ol N_{\otimes}(r, 0; F')+ \ol N_{\otimes}(r, 0; G')- N_0(r,0;F')- N_0(r,0;G')\\&&+ S(r,f)+ S(r,g)\\&\leq& \bigg(\frac{k}{2}+2\bigg)\bigg[\ol N(r, \infty;f)+\ol N(r, \infty;g)\bigg]+ N_{k+2}(r, 0; f)+N_{k+2}(r, 0; g)\\&&-\bigg(l-\frac{3}{2}\bigg) \ol N_*(r, 1; F, G)+ S(r,f)+ S(r,g). \eeas
\end{proof}
\begin{lem} \label{l11}
	Let $f$, $g$ be two transcendental meromorphic functions and $F$, $G$ be defined as in {\it Lemma \ref{l10}}. Then either $f^{(k)}g^{(k)}\equiv Q^2$ or $f \equiv g$, whenever $f$ and $g$ satisfies one of the following conditions:\\ 
	(i) $l\geq 2$ and \bea \label{e2.5}  \bigg(\frac{k}{2}+2\bigg)\{\Theta(\infty,f)+ \Theta(\infty,g)\}+ \delta_{k+2}(0,f)+\delta_{k+2}(0,g)>k+5;\eea 
	(ii) $l=1$ and \bea \label{e2.6}  &&\bigg(\frac{3k}{4}+\frac{9}{4}\bigg)\{\Theta(\infty,f)+\Theta(\infty,g)\}+ \delta_{k+2}(0,f)+ \delta_{k+2}(0,g)\\&&+ \frac{1}{4}\{\delta_{k+1}(0,f)+\delta_{k+1}(0,g)\}> \frac{3k}{2}+6;\nonumber\eea
	(iii) $l=0$ and \bea \label{e2.7} &&\bigg(2k+\frac{7}{2}\bigg)\{\Theta(\infty,f)+\Theta(\infty,g)\}+\delta_{k+2}(0,f)+\delta_{k+2}(0,g)\\&&+\frac{3}{2}\{\delta_{k+1}(0,f)+\delta_{k+1}(0,g)\}> 4k+ 11. \nonumber\eea
\end{lem}
\begin{proof}
	{\bf \ul {Case 1.}} Let $H\not \equiv 0.$\par 
	We consider the following cases.\par 
	{\bf \ul{Subcase 1.1.}} Let $l \geq 2$. Then from {\it Lemma \ref{l10}}, we get 
	\beas &&\frac{1}{2}[T(r, f)+ T(r,g)]\\ &\leq& \bigg(\frac{k}{2}+2\bigg)\bigg[\ol N(r, \infty;f)+\ol N(r, \infty;g)\bigg]+ N_{k+2}(r, 0; f)+N_{k+2}(r, 0; g)\\&&+ S(r,f)+ S(r,g)\\ &\leq& \bigg[\bigg(\frac{k}{2}+3\bigg)- \bigg(\frac{k}{2}+2\bigg)\Theta(\infty,f)- \delta_{k+2}(0,f)\bigg]T(r,f)\\&&+\bigg[\bigg(\frac{k}{2}+3\bigg)-\bigg(\frac{k}{2}+2\bigg)\Theta(\infty,g)- \delta_{k+2}(0,g)\bigg] T(r,g)+S(r,f)+ S(r,g),  \eeas
	i.e., 
	\beas &&\bigg[\bigg(\frac{k}{2}+2\bigg)\Theta(\infty,f)+ \delta_{k+2}(0,f)-\bigg(\frac{k}{2}+\frac{5}{2}\bigg)\bigg]T(r,f)\\&&+  \bigg[\bigg(\frac{k}{2}+2\bigg)\Theta(\infty,g)+ \delta_{k+2}(0,g)-\bigg(\frac{k}{2}+\frac{5}{2}\bigg)\bigg]T(r,g)\leq S(r,f)+ S(r,g).\eeas
	Without loss of generality, we may suppose that there exists a set $I$ with infinite linear measure such that 
	\beas T(r,g)\leq T(r,f),\;\;\;\;\;\;\; r\in I.\eeas
	Then for $r\in I$, we have
	\beas &&\bigg[\bigg(\frac{k}{2}+2\bigg)\{\Theta(\infty,f)+ \Theta(\infty,g)\}+ \delta_{k+2}(0,f)+\delta_{k+2}(0,g)-(k+5)\bigg]T(r,f)\leq S(r,f),\eeas
	which contradicts (\ref{e2.5}). \par 
	{\bf \ul {Subcase 1.2.}} Let $l=1$. So from {\it Lemma, \ref{l4}, \ref{l4a} and \ref{l10}}, we have
	\beas &&\frac{1}{2}[T(r, f)+ T(r,g)]\\&\leq& \bigg(\frac{k}{2}+2\bigg)\bigg[\ol N(r, \infty;f)+\ol N(r, \infty;g)\bigg]+ N_{k+2}(r, 0; f)+N_{k+2}(r, 0; g)\\&&+\frac{1}{2} \ol N_*(r, 1; F, G)+ S(r,f)+ S(r,g)\\ &\leq& \bigg(\frac{k}{2}+\frac{9}{4}\bigg)\bigg[\ol N(r, \infty;f)+\ol N(r, \infty;g)\bigg]+ N_{k+2}(r, 0; f)+N_{k+2}(r, 0; g)\\&&+\frac{1}{4}\{\ol N(r,0;F)+  \ol N(r,0;G)\} + S(r,f)+ S(r,g)\\ &\leq& \bigg(\frac{3k}{4}+\frac{9}{4}\bigg)\bigg[\ol N(r, \infty;f)+\ol N(r, \infty;g)\bigg]+ N_{k+2}(r, 0; f)+N_{k+2}(r, 0; g)\\&&+\frac{1}{4}\{N_{k+1}(r,0;f)+  N_{k+1}(r,0;g)\} + S(r,f)+ S(r,g)\\ &\leq& \bigg[\bigg(\frac{3k}{4}+\frac{7}{2}\bigg)- \bigg(\frac{3k}{4}+\frac{9}{4}\bigg)\Theta(\infty,f)- \delta_{k+2}(0,f)- \frac{1}{4}\delta_{k+1}(0,f)\bigg]T(r,f)\\&&+\bigg[\bigg(\frac{3k}{4}+\frac{7}{2}\bigg)- \bigg(\frac{3k}{4}+\frac{9}{4}\bigg)\Theta(\infty,g)- \delta_{k+2}(0,g)- \frac{1}{4}\delta_{k+1}(0,g)\bigg]T(r,g)\\&& +S(r,f)+ S(r,g),  \eeas
	i.e., \beas&& \bigg[\bigg(\frac{3k}{4}+\frac{9}{4}\bigg)\Theta(\infty,f)+ \delta_{k+2}(0,f)+ \frac{1}{4}\delta_{k+1}(0,f)-\bigg(\frac{3k}{4}+3\bigg)\bigg]T(r,f)\\&&+\bigg[ \bigg(\frac{3k}{4}+\frac{9}{4}\bigg)\Theta(\infty,g)+ \delta_{k+2}(0,g)+ \frac{1}{4}\delta_{k+1}(0,g)-\bigg(\frac{3k}{4}+3\bigg)\bigg]T(r,g)\\&\leq& S(r,f)+ S(r,g). \eeas
	Without loss of generality, we may suppose that there exists a set $I$ with infinite linear measure such that 
	\beas T(r,g)\leq T(r,f),\;\;\;\;\;\;\; r\in I.\eeas
	Then for $r\in I$, we have
	\beas &&\bigg[\bigg(\frac{3k}{4}+\frac{9}{4}\bigg)\{\Theta(\infty,f)+\Theta(\infty,g)\}+ \delta_{k+2}(0,f)+ \delta_{k+2}(0,g)+ \frac{1}{4}\{\delta_{k+1}(0,f)+\delta_{k+1}(0,g)\}\\&&-\bigg(\frac{3k}{2}+6\bigg)\bigg] T(r,f)\leq S(r,f), \eeas
	which contradicts (\ref{e2.6}).\par 
	{\bf \ul {Subcase 1.3.}} Let $l=0$. So from {\it Lemma, \ref{l4}, \ref{l4a} and \ref{l10}}, we have
	\beas &&\frac{1}{2}[T(r, f)+ T(r,g)]\\ &\leq& \bigg(\frac{k}{2}+2\bigg)\bigg[\ol N(r, \infty;f)+\ol N(r, \infty;g)\bigg]+ N_{k+2}(r, 0; f)+N_{k+2}(r, 0; g)\\&&+\frac{3}{2} \ol N_*(r, 1; F, G)+ S(r,f)+ S(r,g)\\ &\leq& \bigg(\frac{k}{2}+\frac{7}{2}\bigg)\bigg[\ol N(r, \infty;f)+\ol N(r, \infty;g)\bigg]+ N_{k+2}(r, 0; f)+N_{k+2}(r, 0; g)\\&&+\frac{3}{2}\{\ol N(r,0;F)+  \ol N(r,0;G)\} + S(r,f)+ S(r,g)\\ &\leq& \bigg(2k+\frac{7}{2}\bigg)\bigg[\ol N(r, \infty;f)+\ol N(r, \infty;g)\bigg]+ N_{k+2}(r, 0; f)+N_{k+2}(r, 0; g)\\&&+\frac{3}{2}\{N_{k+1}(r,0;f)+  N_{k+1}(r,0;g)\} + S(r,f)+ S(r,g)\\&\leq& \bigg[(2k+6)-\bigg(2k+\frac{7}{2}\bigg)\Theta(\infty,f)-\delta_{k+2}(0,f)-\frac{3}{2}\delta_{k+1}(0,f)\bigg]T(r,f)\\&&+ \bigg[(2k+6)-\bigg(2k+\frac{7}{2}\bigg)\Theta(\infty,g)-\delta_{k+2}(0,g)-\frac{3}{2}\delta_{k+1}(0,g)\bigg]T(r,g)\\&&S(r,f)+ S(r,g),\eeas
	i.e., 
	\beas &&\bigg[\bigg(2k+\frac{7}{2}\bigg)\Theta(\infty,f)+\delta_{k+2}(0,f)+\frac{3}{2}\delta_{k+1}(0,f)-\bigg(2k+\frac{11}{2}\bigg)\bigg]T(r,f)\\&&+ \bigg[\bigg(2k+\frac{7}{2}\bigg)\Theta(\infty,g)+\delta_{k+2}(0,g)+\frac{3}{2}\delta_{k+1}(0,g)-\bigg(2k+\frac{11}{2}\bigg)\bigg]T(r,g)\\&\leq& S(r,f)+ S(r,g).\eeas
	Without loss of generality, we may suppose that there exists a set $I$ with infinite linear measure such that 
	\beas T(r,g)\leq T(r,f),\;\;\;\;\;\;\; r\in I.\eeas
	Then for $r \in I$, we have \beas &&\bigg[\bigg(2k+\frac{7}{2}\bigg)\{\Theta(\infty,f)+\Theta(\infty,g)\}+\delta_{k+2}(0,f)+\delta_{k+2}(0,g)+\frac{3}{2}\{\delta_{k+1}(0,f)+\delta_{k+1}(0,g)\}\\&&-\bigg(4k+11\bigg)\bigg]T(r,f) \leq S(r,f), \eeas
	which contradicts (\ref{e2.7}).\par 
	{\bf \ul {Case 2.}} Let $H\equiv 0$.
	On integration we get from (\ref{e2.1})
	\bea \label{e2.8} F \equiv \frac{AG+B}{CG+D}, \eea
	where $A$, $B$, $C$, $D$ are complex constants satisfying $AD-BC\not=0$. Also by Mokhon'ko's Lemma \cite{IKU_Mko_71} 
	\bea \label{e2.9} T(r, f)= T(r,g)+ S(r,f).\eea 
	From {\it Lemma \ref{l9a}} we see that $F$, $G$ share $(1, \infty)$ which again implies $F$, $G$ share $(1,2)$. So we consider only the inequality (\ref{e2.5}) and so from (\ref{e2.9}), the condition becomes
	\bea \label{e2.10} \bigg(\frac{k}{2}+2\bigg) \{\Theta(\infty, f)+\Theta(\infty, g)\}+ \delta_{k+2}(0,f)+\delta_{k+2}(0,g) > k+5. \eea
	As $AD-BC\not=0$, so both $A$ and $C$ cannot be simultaneously zero. Thus we consider the following cases.\par 
	{\bf \ul {Subcase 2.1.}} Suppose $AC \not=0$. Then $F-\frac{A}{C}= \frac{BC-AD}{C(CG+D)}\not=0$. So $F$ omits the value $\frac{A}{C}$.
	Now by using {\it Lemma \ref{l5}}, we get
	\beas T(r,f) &\leq& N_{k+1}(r,0;f)+ \ol N(r, \infty;f)+ \ol N\bigg(r, \frac{A}{C}; F\bigg)- N_0(r, 0;F')+ S(r,f)\\&\leq&  N_{k+2}(r, 0;f)+ \ol N(r, \infty;f)+ S(r,f),\eeas
	which yields \beas \delta_{k+2}(0,f)+ \Theta(\infty;f) \leq 1.\eeas
	Thus from (\ref{e2.10}), we get
	\beas \bigg(\frac{k}{2}+1\bigg) \Theta(\infty, f)+ \bigg(\frac{k}{2}+2\bigg)\Theta(\infty, g)+ \delta_{k+2}(0,g)> k+4,\eeas
	which is a contradiction from the {\it Definition \ref{d7}}.\par 
	{\bf \ul {Subcase 2.2.}} Suppose $AC=0$.\par
	{\bf \ul{Subcase 2.2.1.}} Let $A=0$ and $C\not=0$. Then (\ref{e2.8}) becomes 
	\beas  F=\frac{1}{\gamma G+\delta}, \eeas
	where $\gamma= \frac{C}{D}$ and $\delta= \frac{D}{C}$.\par 
	If $F$ has no $1$-point, i.e., $1$ is a Picard Exceptional value, then by using {\it Lemma \ref{l5}}, we get
	\beas T(r,f) &\leq& N_{k+1}(r,0;f)+ \ol N(r, \infty;f)+ \ol N(r, 1, F)- N_0(r, 0;F')+ S(r,f)\\&\leq&  N_{k+2}(r, 0;f)+ \ol N(r, \infty;f)+ S(r,f),\eeas
	which again yields \beas \delta_{k+2}(0,f)+ \Theta(\infty;f) \leq 1,\eeas
	and similarly as above we arrive at a contradiction.\par
	So let $F$ has some $1$-point. Then $\gamma+\delta=1$. Since $C\not =0$, so $\gamma \not=0$ and thus \beas F= \frac{1}{\gamma G+1-\gamma}. \eeas
	By using {\it Lemma \ref{l5}}, we get
	\beas T(r,f) &\leq&  N_{k+1}(r,0;f)+ \ol N(r, \infty;f)+ \ol N\bigg(r, \frac{1}{1-\gamma}, F\bigg)- N_0(r, 0;F')+ S(r,f)\\&\leq& N_{k+1}(r, 0;f)+ \ol N(r, \infty;f)+\ol N(r,0;G)+ S(r,f)+ S(r,g)\\&\leq&  N_{k+1}(r, 0;f)+N_{k+1}(r,0;g)+ \ol N(r, \infty;f)+ k \ol N(r,\infty;g)+S(r,f)+ S(r,g)\\ &\leq&  N_{k+2}(r, 0;f)+N_{k+2}(r,0;g)+ \ol N(r, \infty;f)+ k \ol N(r,\infty;g)+S(r,f)+ S(r,g) ,\eeas
	which yields by using (\ref{e2.9})
	\beas \delta_{k+2}(0,f)+\delta_{k+2}(0,g)+\Theta(\infty,f)+k \Theta(\infty,g) \leq k+2.\eeas
	Thus from (\ref{e2.10}), we get 
	\beas \bigg(\frac{k}{2}+1\bigg)\Theta(\infty,f)+\bigg(2-\frac{k}{2}\bigg)\Theta(\infty,g)> 3, \eeas
		which is a contradiction from the {\it Definition \ref{d7}}.\par 
	Thus $\gamma=1$ and $FG\equiv 1$, i.e., $f^{(k)}g^{(k)}\equiv Q^2$.\par 
	{\bf \ul {Subcase 2.2.2.}} Let $A\not=0$ and $C=0$. Then $D\not=0$ and (\ref{e2.8}) becomes
	\beas F= \lambda G + \mu, \eeas 
	where $\lambda=\frac{A}{D}$ and $\mu=\frac{B}{D}$.\par 
	If $F$ has no $1$-point then the case can be treated similarly as done above.\par
	So let $F$ has some $1$-point. Then $\lambda+ \mu=1$ such that $\mu \not=0$.
	If $\lambda \not=1$ then by using {\it Lemma \ref{l5}}, we get
	\beas T(r,f) &\leq& N_{k+1}(r,0;f)+ \ol N(r,1-\lambda;F)+ \ol N(r, \infty;f)- N_0(r, 0;F')+ S(r,f)\\  &\leq& N_{k+1}(r,0;f)+ \ol N(r,0;G)+ \ol N(r, \infty;f)+ S(r,f)\\ &\leq& N_{k+1}(r,0;f)+ N_{k+1}(r,0;g)+k \ol N(r, \infty, g)+ \ol N(r, \infty;f)+ S(r,f)\\&\leq& N_{k+2}(r,0;f)+ N_{k+2}(r,0;g)+k \ol N(r, \infty, g)+ \ol N(r, \infty;f)+ S(r,f),\eeas
	which again yields by using (\ref{e2.9})
	\beas \delta_{k+2}(0,f)+\delta_{k+2}(0,g)+\Theta(\infty,f)+k \Theta(\infty,g) \leq k+2,\eeas
    and similarly as above we arrive at a contradiction.\par 
	Thus $\lambda=1$ and so $F\equiv G$, which can be rewritten as 
	\bea \label{e2.11} f = g + Q_1,\eea
	where $Q_1$ is a polynomial with degree $\gamma_{Q_1}\leq k-1$. Combining (\ref{e2.11}) and Nevanlinna's three small functions theorem (Theorem 1.36, \cite{KAP_YangYi_03}) we get 
	\bea \label{e2.12} T(r,g)&\leq& \ol N(r, \infty;g)+ \ol N(r,0;g)+ \ol N(r, 0; g+Q_1)+ S(r,g)\\ &=& \ol N(r, \infty;g)+ \ol N(r, 0;g)+ \ol N(r, 0; f)+ S(r,g).\nonumber\eea
	Again form (\ref{e2.11})  we get $T(r,f)= T(r,g)+ O(\log r)$. From this and (\ref{e2.12}) we get 
	\bea \label{e2.13} \Theta(0,f)+ \Theta(0,g)+ \Theta(\infty,g) \leq 2. \eea
	From (\ref{e2.10}), (\ref{e2.13}) and {\it Remark \ref{r1}}, we get
	\bea \label{e2.14} \bigg(\frac{k}{2}+2\bigg) \Theta(\infty, f)+ \bigg(\frac{k}{2}+1\bigg)\Theta(\infty, g)> k+3.\eea
	Hence from (\ref{e2.14}) and {\it Remark \ref{r1}} we get a contradiction. Thus $Q_1=0$ and so we get from (\ref{e2.11}) that $f=g$. This completes the proof. 
\end{proof}
\section{Proof of the Theorems}
\begin{proof}[\bf {Proof of Theorem \ref{t1}}]
Let $d$ be the degree of the $L$-function $\mathcal{L}$. Therefore, by Steuding \cite[p. 150]{LNM_Ste_07} we have 
\bea \label{e3.1} T(r,\mathcal{L}) = \frac{d}{\pi} r \log r+ O(r).\eea	
We set the functions $F_1$ and $G_1$ as follows.
\bea \label{e3.2} F_1= \frac{F^{(k)}}{\eta(z)}, \;\;\; G_1=  \frac{G^{(k)}}{\eta(z)},\eea
where $F= P(f)(\alpha f^m+ \beta)^s$ and $G=P(\mathcal{L})(\alpha \mathcal{L}^m+ \beta)^s$.\par
Clearly as $F^{(k)}-\eta(z)$, $G^{(k)}-\eta(z)$ share $(0,l)$, hence $F_1$, $G_1$ share $(1,l)$.\par 
Noting that an $L$-function has at most one pole $z=1$ in the complex plane, we deduce by {\it Lemma \ref{l1}, \ref{l7}} and Valiron-Mokhonko's lemma (cf. \cite{IKU_Mko_71}) that 
\beas (n+ms)T(r,\mathcal{L})+ S(r,f)&=& T(r,G) \\ &\leq& \ol N(r,\infty; G)+ N(r,0; G)+ \ol N(r, 1; G_1)- N(r, 0; G_1^{'})+ S(r,f)\\ &\leq& \ol N(r,\infty; G)+ N_{k+1}(r,0; G)+ \ol N(r, 1; G_1)- N_0(r, 0; G_1^{'})+ S(r,f)\\&\leq& \ol N(r,\infty; \mathcal{L})+(k+1)\ol N(r,0;G) + \ol N(r, 1; G_1)+ S(r,f)\\ &\leq& (k+1)(n+ms) T(r,\mathcal{L}) + \ol N(r, 1; F_1)+ S(r,f), \eeas
where $N_0(r,0; G_1^{'})$ is the counting function of those zeros of $G_1^{'}$ in $|z|<r$ which are not the zeros of $G$ and $G_1-1$ in $|z|<r$.
This implies \bea \label{e3.3} -k(n+ms)T(r,\mathcal{L})\leq T(r,F^{(k)})+ S(r,f).\eea
By (\ref{e3.1}) we see that $\mathcal{L}$ is a transcendental meromorphic function. Combining this with (\ref{e3.3}),  \cite[Theorem 1.5]{KAP_YangYi_03} and the assumption of the lower bound of $n$, we deduce that $F^{(k)}$ and so $f$ is a transcendental meromorphic function. \par
Now we set \bea \label{e3.4}  \Delta_{1}= \bigg(\frac{k}{2}+2\bigg)\{\Theta(\infty,F)+ \Theta(\infty,G)\}+ \delta_{k+2}(0,F)+\delta_{k+2}(0,G),\eea
\bea \label{e3.5} \Delta_{2} &=& \bigg(\frac{3k}{4}+\frac{9}{4}\bigg)\{\Theta(\infty,F)+\Theta(\infty,G)\}+ \delta_{k+2}(0,F)+ \delta_{k+2}(0,G)\\&&+ \frac{1}{4}\{\delta_{k+1}(0,F)+\delta_{k+1}(0,G)\}\nonumber\eea
and 
\bea \label{e3.6} \Delta_3&=& \bigg(2k+\frac{7}{2}\bigg)\{\Theta(\infty,F)+\Theta(\infty,G)\}+\delta_{k+2}(0,F)+\delta_{k+2}(0,G)\\&&+\frac{3}{2}\{\delta_{k+1}(0,F)+\delta_{k+1}(0,G)\}.\nonumber\eea
Using {\it Lemma \ref{l1}}, we have
\bea \label{e3.7} \Theta(\infty, F) &=& 1- \limsup\limits_{r \to \infty}\frac{\ol N(r, \infty; F)}{T(r,F)}\\&&=1- \limsup\limits_{r \to \infty}\frac{\ol N(r, \infty; f)}{(n+ms)T(r,f)+O(1)}\geq 1- \frac{1}{n+ms},\nonumber \eea 
\bea \label{e3.8} \delta_{k+2}(0, F) &=& 1- \limsup\limits_{r \to \infty}\frac{N_{k+2}(r, 0; F)}{T(r,F)}\\&\geq& 1- \limsup\limits_{r \to \infty}\frac{ N_{k+2}(r,0;P(f))+  N_{k+2}(r,0; (\alpha f^m+\beta)^s)}{(n+ms)T(r,f)+O(1)}\nonumber\\&\geq& 1- \frac{n_2(k+2) +n_1+ms}{n+ms}.\nonumber \eea
Similarly,  
\bea \label{e3.9} \delta_{k+2}(0,G)\geq 1- \frac{n_2(k+2) +n_1+ms}{n+ms}, \eea
\bea \label{e3.10} \delta_{k+1}(0,F)\geq 1- \frac{n_2(k+1) +n_1+ms}{n+ms}, \eea
\bea \label{e3.11} \delta_{k+1}(0,G)\geq 1- \frac{n_2(k+1) +n_1+ms}{n+ms}. \eea 
Since an $L$-function has at most one pole at $z=1$ in the complex plane, we have 
\beas N(r,\mathcal{L})\leq \log r+ O(1).\eeas
So using (\ref{e3.1}) we deduce that 
\bea \label{e3.12} \Theta(\infty, G)=1. \eea
{\ul {\bf Case 1.}} Let  $l \geq 2$.\par 
By using (\ref{e3.4}), (\ref{e3.7})-(\ref{e3.9}) and (\ref{e3.12}), we have
\bea \label{e3.13} \Delta_1 &\geq& (k+6)- \frac{(\frac{k}{2}+2)+2n_2(k+2) +2n_1+2ms}{n+ms}.\eea
By (\ref{e3.13}) and the assumption $n>(\frac{k}{2}+2)+2n_2(k+2) +2n_1+ms$, we have $\Delta_1>k+5$. Thus by {\it Lemma \ref{l11}} we get either $F^{(z)} G^{(k)}= \eta^2(z)$ or $F \equiv G$.\par 
Let $F\equiv G$, i.e.,  \bea \label {e3.14} P(f)(\alpha f^m+\beta)^s=P(\mathcal{L})(\alpha \mathcal{L}^m+\beta)^s.\eea
Now we set \bea \label{e3.15} H= \frac{f}{\mathcal{L}}.\eea 
If $H$ is a non-constant meromorphic function, then we get (\ref{e3.14}). \par
Suppose $H$ is a constant. Then from (\ref{e3.15}), we get
	\beas &&[a_nf^n+ a_{n-1}f^{n-1}+ \ldots+ a_1z] \bigg[(\alpha f^m)^s+ {s \choose 1}(\alpha f^m)^{s-1}\beta+\ldots+{s \choose s}\beta^s\bigg]\\&&= (a_n\mathcal{L}^n+ a_{n-1}\mathcal{L}^{n-1}+ \ldots+ a_1z)\bigg[(\alpha \mathcal{L}^m)^s+ {s \choose 1}(\alpha \mathcal{L}^m)^{s-1}\beta+\ldots+{s \choose s}\beta^s\bigg], \eeas
	i.e., \beas &&\sum\limits_{i=0}^{s}{s \choose i}\beta^i[a_n \mathcal{L}^{m(s-i)+n}(H^{m(s-i)+n}-1)+a_{n-1} \mathcal{L}^{m(s-i)+n-1}(H^{m(s-i)+n-1}-1)+\ldots\\&&+a_1 \mathcal{L}^{m(s-i)+1}(H^{m(s-i)+1}-1)]=0, \eeas
	which implies $H^{\chi_n}=1$, where 
	
	\beas \chi_n =\begin{cases}
		1, &\displaystyle\sum_{j=1}^{n-1}\mid a_{n-j}\mid\not=0; \\ d_1, & a_j=0, \forall j=1,2,\ldots, n-1,	\end{cases}\eeas
	$d_1=\gcd(ms+n,\ldots, m(s-i)+n, \ldots n)$, $i=0,1,\ldots,s$. Therefore, $f =t \mathcal{L}$,  for a constant $t$ satisfying $t^{\chi_n}=1$.\par 
{\ul {\bf Case 2.}} Let  $l =1$.\par 
By using (\ref{e3.5}), (\ref{e3.7})-(\ref{e3.12}), we have
\bea \label{e3.16} \Delta_2 &\geq& \frac{3k}{2}+7-\frac{\frac{3k}{4}+\frac{9}{4}+\bigg(\frac{5k}{2}+\frac{9}{2}\bigg)n_2 +\frac{5n_1}{2}+\frac{5ms}{2}}{n+ms}. \eea
By (\ref{e3.16}) and the assumption $n>\frac{3k}{4}+\frac{9}{4}+\bigg(\frac{5k}{2}+\frac{9}{2}\bigg)n_2 +\frac{5n_1}{2}
+\frac{3ms}{2}$, we have $\Delta_2> \frac{3k}{2}+6$. Thus by {\it Lemma \ref{l11}} we get either $F^{(z)} G^{(k)}= \eta^2(z)$ or $F \equiv G$. Proceeding in the same manner as done in {\it Case 1}, we get the conclusion.\par
{\ul {\bf Case 3.}} Let  $l=0$.\par 
By using (\ref{e3.6}), (\ref{e3.7})-(\ref{e3.12}), we have
\bea \label{e3.17} \Delta_3 &\geq& 4k+12-\frac{2k+\frac{7}{2}+(5k+7)n_2 +5n_1+5ms}{n+ms}.\eea
By (\ref{e3.17}) and the assumption $n>2k+\frac{7}{2}+(5k+7)n_2 +5n_1+4ms$, we have $\Delta_3> 4k+11$. Thus by {\it Lemma \ref{l11}} we get either $F^{(z)} G^{(k)}= \eta^2(z)$ or $F \equiv G$. Proceeding in the same manner as done in {\it Case 1}, we get the conclusion.
\end{proof}

\begin{proof}[\bf{Proof of Corollary \ref{c2}}]
	We set the functions $F_1$ and $G_1$ as follows.
	\bea \label{e3.18} F_1= \frac{F^{(k)}}{\eta(z)}, \;\;\; G_1=  \frac{G^{(k)}}{\eta(z)},\eea
	where $F= f^n(\alpha f^m+ \beta)$ and $G=\mathcal{L}^n(\alpha \mathcal{L}^m+ \beta)$.\par
	Clearly as $F^{(k)}-\eta(z)$, $G^{(k)}-\eta(z)$ share $(0,l)$, hence $F_1$, $G_1$ share $(1,l)$.\par
Then using the same procedure as adopted in {\it Theorem \ref{t1}}, we obtain either $F^{(z)} G^{(k)}= \eta^2(z)$ or $F \equiv G$.\par
{\bf \ul {Subcase 1.1.}} Suppose $F^{(k)} G^{(k)}= \eta^2(z)$, i.e., \bea \label{e3.19} \{f^n(\alpha f^m+ \beta)\}^{(k)} \{\mathcal{L}^n(\alpha \mathcal{L}^m+ \beta)\}^{(k)}=\eta^2(z).\eea
{\bf \ul {Subcase 1.1.1.}} Let $\alpha\beta \not=0$.\par Then using (\ref{e3.1}),  (\ref{e3.19}), {\it Lemma \ref{l1}} and a result from Whittaker \cite[p.82]{JLMS_Whi_36} and the definition of the order of a meromorphic function we have
\bea \label{e3.20} \rho(f)&=& \rho((f^n(\alpha f^m+ \beta))^{(k)})= \rho\bigg(\frac{(f^n(\alpha f^m+ \beta))^{(k)}}{\eta^2(z)}\bigg)\\&=&\rho((\mathcal{L}^n(\alpha \mathcal{L}^m+ \beta))^{(k)})= \rho(\mathcal{L}^n(\alpha \mathcal{L}^m+ \beta))= \rho(\mathcal{L})=1.\nonumber\eea
By (\ref{e3.20}) we can see that $f$ is a transcendental meromorphic function. Since an $L$-function has at most one pole at $z=1$ in the complex plane, we deduce by (\ref{e3.19}) that $\frac{(f^n(\alpha f^m+ \beta))^{(k)}}{\eta^2(z)}$ has at most one zero at $z=1$ in the complex plane. Now as $\eta(z)$ is a polynomial so the zero comes only from $(f^n(\alpha f^m+ \beta))^{(k)}$. Combining this with (\ref{e3.20}), {\it Lemma \ref{l8}} and the assumption $k\geq 2$, we obtain that $f^n(\alpha f^m+ \beta)$ has finitely many poles and so $f$ has finitely many poles in the complex plane. This together with (\ref{e3.19}) implies that $\frac{(\mathcal{L}^n(\alpha \mathcal{L}^m+ \beta))^{(k)}}{\eta^2(z)}$ and so $(\mathcal{L}^n(\alpha \mathcal{L}^m+ \beta))^{(k)}$ has at most finitely many zeros in the complex plane. Moreover, by the assumptions  (\ref{e1.8}) - (\ref{e1.10}), we deduce that $\mathcal{L}$ has at most finitely many zeros. Thus, 
\bea \label{e3.21} \mathcal{L}= R_1 e ^{A_1 z + B_1},\eea
where $R_1$ is a rational function, $A_1\not=0$ and $B_1$ are constants. By (\ref{e3.21}) and Hayman \cite[p.7]{Clarendon_Hay_64} we have 
\bea \label{e3.22} T(r, \mathcal{L})&=& T(r, R_1 e ^{A_1 z + B_1})\\ &\leq& \frac{|A_1|r}{\pi}(1+o(1))+ O(\log r),\nonumber\eea 
which contradicts (\ref{e3.1}).\par 
{\bf {\ul {Subcase 1.1.2.}}} Let $\alpha \beta =0$. As $\mid \alpha \mid + \mid \beta \mid \not = 0$, we have to consider the following two cases.\par
{\bf {\ul {Subcase 1.1.2.1}}} Let $\alpha \not=0$, $\beta=0$. Then (\ref{e3.19}) becomes $\{f^{n+m}\alpha^{m}\}^{(k)} \{\mathcal{L}^{n+m}\alpha^{m}\}^{(k)}=\eta^2(z)$. 
Let $z_0$ be a zero of $\mathcal{L}$ of order $\lambda$. Then we can get that $z_0$ is a pole of $f$ of order $\chi$, satisfying $(n+m)\lambda-k= (n+m)\chi+k$, i.e., $(n+m)(\lambda-\chi)=2k$, which implies $n+m \leq 2k$, contradicting the assumptions (\ref{e1.8}) - (\ref{e1.10}). Hence $\mathcal{L}$ has no zeros and so 
\beas \mathcal{L}= R_2 e^{A_2 z+ B_2},\eeas
where $R_2$ is a rational function and $A_2(\not=0)$, $B_2$ are constants. Thus adopting the same procedure as in {\it Subcase 1.1.1} we arrive at a contradiction. \par 
{\bf {\ul {Subcase 1.1.2.2}}} Let $\alpha =0$, $\beta\not=0$. Then (\ref{e3.19}) becomes $\{f^n\beta^{m}\}^{(k)} \{\mathcal{L}^n\beta^{m}\}^{(k)}=\eta^2(z)$. By using the argument as in {\it Subcase 1.1.2.1}, we obtain $n(\lambda-\chi)=2k$, which again contradicts the assumptions (\ref{e1.8}) - (\ref{e1.10}). Thus in the similar way we arrive at a contradiction as in {\it Subcase 1.1.2.1}.\par
{\bf \ul {Subcase 1.2.}} Let $F \equiv G$, i.e., 
\bea \label{e3.22a}f^n(\alpha f^m+ \beta)=\mathcal{L}^n(\alpha \mathcal{L}^m+ \beta).\eea
 So $f$ and $\mathcal{L}$ share $(\infty,\infty)$.\par
{\bf \ul{Subcase 1.2.1.}} Let $\alpha \beta\not =0$. \par
 
Taking $H= \frac{f}{\mathcal{L}}$, we get 
\bea \label{e3.23} \alpha  \mathcal{L}^{n+m} (H^{n+m}-1)= -\beta \mathcal{L}^n (H^n-1). \eea
Suppose $H$ is a non-constant meromorphic function. Then by (\ref{e3.23}) we have 
\bea \label{e3.24} \frac{\alpha \mathcal{L}^m}{\beta}= -\frac{H^n-1}{H^{n+m}-1}. \eea
Let $d= \gcd(n,m)$. Then clearly $H^d=1$ is the common factor of $H^{n}-1$ and $H^{n+m}-1$. Therefore, (\ref{e3.24}) can be rewritten as 
\bea \label{e3.25} \frac{\alpha \mathcal{L}^m}{\beta}=- \frac{1+H+\ldots+H^{n-d}}{1+H+\ldots+H^{n+m-d}}. \eea
By (\ref{e3.25}) and {\it Lemma \ref{l1}} we have 
\bea \label{e3.26} T(r,\mathcal{L})= T\bigg(r, \frac{1+H+\ldots+H^{n-d}}{1+H+\ldots+H^{n+m-d}}\bigg)= (n+m-d) T(r,H)+ O(1).\eea
Also \bea \label{e3.27} &&\rho(f)= \rho(f^n (\alpha f^m+ \beta))=\rho((f^n (\alpha f^m+ \beta))^{(k)})=\rho((\mathcal{L}^n (\alpha \mathcal{L}^m+ \beta))^{(k)})\\&& =\rho(\mathcal{L}^n (\alpha \mathcal{L}^m+ \beta))= \rho(\mathcal{L})=1.\nonumber\eea
By (\ref{e3.23}), (\ref{e3.26}), (\ref{e3.27}) and by the second fundamental theorem we have
\bea \label{e3.28} \ol N(r, \infty;\mathcal{L})= \sum\limits_{j=1}^{n} \ol N(r, \gamma_j; H)+ o(T(r,H))\geq (n+m-d-1) T(r,H),\eea 
as $r \to \infty$. Here $\lambda_1, \lambda_2, \ldots, \lambda_{n+m-d}$ are $n+m-d$ distinct finite complex numbers satisfying $\lambda_j \not=1$ and $\lambda_j^{n+m-d}=1$ for $1 \leq j \leq n+m-d$. Noting that $\mathcal{L}$ is a transcendental meromorphic function such that $\mathcal{L}$ has at most one pole $z=1$ in the complex plane, we deduce by (\ref{e3.28}) that there exists some small positive number $\epsilon$ satisfying $0 < \epsilon < 1$, such that 
\bea \label{e3.29} (n+m-d-1-\epsilon) T(r,H) \leq \ol N(r, \infty;\mathcal{L})= \log r + O(1). \eea 
By (\ref{e3.29}) and the assumptions (\ref{e1.8}) - (\ref{e1.10}) and $k \geq 2$ we deduce that $H$ is a non-constant rational function.

Since $f$ and $\mathcal{L}$ share $(\infty,\infty)$, it follows from the construction of $H$, that the poles of $H$ only comes from the zeros of $\mathcal{L}$ and so they are finite in number. As a result $\mathcal{L}$ has a representation  
\beas \mathcal{L}= R_3 e^{A_3 z+ B_3},\eeas
where $R_3$ is a rational function and $A_3(\not=0)$, $B_3$ are constants. Thus proceeding the similar way as adopted in {\it Subcase 1.1.1} we arrive at a contradiction. 

When $H$ is a constant meromorphic function then from (\ref{e3.22a}), we get
\beas\alpha \mathcal{L}^{n+m}(H^{n+m}-1)+ \beta \mathcal{L}^n(H^n-1)=0,\eeas
which implies $H^d=1$. Therefore, $f=t\mathcal{L}$, for a constant $t$ satisfying $t^d=1$.\par 
{\bf \ul{Subcase 1.2.2.}} Let $\alpha \beta =0$. As $\mid \alpha \mid + \mid \beta \mid \not =0$, so we have to consider the following two cases. \par
{\bf \ul {Subcase 1.2.2.1.}} Let $\alpha=0$ and $\beta \not=0$. Then clearly we get $f = t \mathcal{L}$, where $t$ is a constant satisfying $t^n=1$.\par 
{\bf \ul {Subcase 1.2.2.2.}} Let $\alpha\not=0$ and $\beta =0$. Then clearly we see that $f = t \mathcal{L}$, where $t$ is a constant satisfying $t^{n+m}=1$.
\end{proof}


\begin{thebibliography}{99}
	\bibitem{RCMP_BanBha_16} A. Banerjee and S. Bhattacharyya, On the uniqueness of meromorphic functions and its difference operator sharing values or sets, Rend. Circ. Mat Palermo, II, Ser., DOI. 10.1007/s12215-016-0295-1. 
	\bibitem{TJM_Ban_10} A. Banerjee, Uniqueness of meromorphic functions that share two sets with finite weight II, Tamkang J. Math., 41(4) (2010), 379-392. 
	\bibitem{Clarendon_Hay_64} W. K. Hayman, Meromorphic Functions, The Clarendon Press, Oxford (1964).
	\bibitem{DDNS_HaoChe_18} W. J. Hao and J. F. Chen, Uniqueness of $L$-functions concerning certain differential polynomials, Discrete Dyn. Nat. Soc., 2018, DOI. 10.1155/2018/4673165.
	\bibitem{YMJ_Lah_97} I. Lahiri, Uniqueness of meromorphic functions as governed by their differential polynomials, Yokohama Math. J., 44 (1997), 147-156.
	\bibitem{NMJ_Lah_01} I. Lahiri, Weighted sharing and uniqueness of meromorphic functions, Nagoya Math. J., 161 (2001), 193-206.
	\bibitem{CVTA_Lah_01} I. Lahiri, Weighted value sharing and uniqueness of meromorphic functions, Complex Var. Theo. Appl., 46 (2001), 241-253.
	\bibitem{IJMS_Lah_01} I. Lahiri, Value distribution of certain differential polynomials, Int. J. Math. Math. Sci., 28 (2001), 83-91.
	\bibitem{BLMS_Lan_03} J. K. Langley, The second derivative of a meromorphic function of finite order, Bull. Lond. Math. Soc., 35(1) (2003), 97-108.
	\bibitem{Filomat_LiLiuYi_19} X. Li, F. Liu and H. X. Yi, Results on $L$-functions of certain differential polynomial sharing one finite value, Filomat 33(18) (2019), 5767-5776.
	\bibitem{PJASA_LiuLiYi_17} F. Liu, X. M. Li and H. X. Yi, Value distribution of $L$-functions concerning shared values and certain differential polynomials, Proc. Japan. Acad. Ser. A, 93 (2017), 41-46.
	\bibitem{IKU_Mko_71} A. Z. Mokhon'ko, The Nevanlinna characteristics of certain meromorphic functions, Theory of funct. funct. anal. appl., 14(1971), 83-87. 
	\bibitem{TMJ_SahHal_18} P. Sahoo and S. Haldar, Uniqueness results related to $L$-functions and certain differential polynomials, Tbilisi Math. J., 11(4) (2018), 67-78. 
	\bibitem{PACANT_Sel_92} A. Selberg, Old and new conjectures and results about a class of Dirichlet series, Proc. of the Amal Conference on Analytic Number Theory (Maiori, 1989), Univ. Salerno, Salerno, 1992.
	\bibitem{LNM_Ste_07} J. Steuding, Value-distribution of L-functions, Lecture Notes in Math., Springer, Berlin, 2007.
	\bibitem{JLMS_Whi_36} J. M. Whittaker, The order of the derivative of a meromorphic function, J. Lond. Math. Soc., S1-11, 4 (1936), 82-87.
	\bibitem{MathZ_Yang_72} C. C. Yang, On deficiencies of differential polynomials II, Math. Z., 125(1972), 107-112.
	\bibitem{KAP_YangYi_03} C. C. Yang and H. X. Yi, Uniqueness Theory of Meromorphic functions, Kluwer Academic Publishers, Dordrecht, 2003.
	\bibitem{SSSA_Yan_86} L. Yang, Normality for families of meromorphic functions, Sci. Sinica Ser. A, 29 (1986), 1263-1274.
	\bibitem{KMJ_Yi_99} H. X. Yi, Meromorphic functions that share one or two values II, Kodai Math. J., 22(2) (1999), 264-272.
\bibitem{JIPAM_ZhaYan_07} J. L. Zhang and L. Z. Yang, Some results related to a conjecture of R. Bruck, J. Inequal. Pure Appl. Math., 8(1) (2007).
\end{thebibliography}
\end{document}